\documentclass[12pt,reqno]{amsart}
\usepackage{fullpage}
\allowdisplaybreaks
\usepackage{times}
\usepackage{colonequals}
\usepackage{amsmath,amssymb,amsthm,url}
\usepackage[utf8]{inputenc}
\usepackage[english]{babel}
\usepackage[alphabetic]{amsrefs}
\usepackage{bbm}
\usepackage{enumerate}
\usepackage{bm}
\usepackage{graphicx}
\usepackage{mathrsfs}
\usepackage[colorlinks=true, pdfstartview=FitH, linkcolor=blue, citecolor=blue, urlcolor=blue]{hyperref}
\newtheorem{thm}{Theorem}[section]
\newtheorem{lem}[thm]{Lemma}
\newtheorem{prop}[thm]{Proposition}

\theoremstyle{definition}

\newtheorem{question}[thm]{Question}
\newtheorem{rem}[thm]{Remark}

\newtheorem{ex}[thm]{Example}

\newcommand{\bP}{\mathbb{P}}

\newcommand{\F}{\mathbb{F}}
\newcommand{\Fq}{\mathbb{F}_q}

\AtBeginDocument{%
	\def\MR#1{}
}

\begin{document}

\title{Proportion of blocking curves in a pencil}

\author{Shamil Asgarli}
\address{Department of Mathematics and Computer Science \\ Santa Clara University \\ 500 El Camino Real \\ USA 95053}
\email{sasgarli@scu.edu}

\author{Dragos Ghioca}
\address{Department of Mathematics \\ University of British Columbia \\ 1984 Mathematics Road \\ Canada V6T 1Z2}
\email{dghioca@math.ubc.ca}

\author{Chi Hoi Yip}
\address{School of Mathematics\\ Georgia Institute of Technology\\ Atlanta, GA 30332\\ United States}
\email{cyip30@gatech.edu}


\begin{abstract}
Let $\mathcal{L}$ be a pencil of plane curves defined over $\mathbb{F}_q$ with no $\F_q$-points in its base locus. We investigate the number of curves in $\mathcal{L}$ whose $\F_q$-points form a blocking set. When the degree of the pencil is allowed to grow with respect to $q$, we translate the geometric problem into a purely combinatorial problem about disjoint blocking sets. We also study the same problem when the degree of the pencil is fixed.
\end{abstract}

\maketitle

\section{Introduction}

Throughout the paper, $p$ denotes a prime, $q$ denotes a power of $p$, and $\F_q$ denotes the finite field with $q$ elements. Recall that a set of points $B\subseteq \bP^2(\F_q)$ is a \emph{blocking set} if every $\F_q$-line meets $B$. For example, a union of $q+1$ distinct $\F_q$-points on a line forms a blocking set. A blocking set $B$ is \emph{trivial} if it contains all the $q+1$ points of a line, and is \emph{nontrivial} if it is not a trivial blocking set. Blocking sets have been studied extensively in finite geometry and design theory.

Inspired by the rich interaction between finite geometry and algebraic curves \cite{SS98}, the concept of blocking curves was formally introduced in \cite{AGY23}. Given a projective plane curve $C\subset\mathbb{P}^2$ defined over $\F_q$, recall that $C(\F_q)$ denotes the set of $\F_q$-rational points on $C$. We say that $C$ is a \emph{blocking curve} if $C(\F_q)$ is a blocking set; otherwise, it is \emph{nonblocking}. Moreover, $C$ is \emph{nontrivially blocking} if $C(\F_q)$ is a nontrivial blocking set.

In our previous papers \cite{AGY23} and \cite{AGY22b}, we showed that irreducible blocking curves of low degree $d\geq 2$ (namely, satisfying $d^6 < q$) do not exist, and that blocking curves of high degree are rare. In other words, our past results suggest that a random curve over $\F_q$ is very likely to be nonblocking. In \cite{AGY23-FFA-pencils}, we also proved that there are pencils of plane curves in which \emph{every} $\mathbb{F}_q$-member is nonblocking. These results rely on combinatorial properties of blocking sets as well as tools from algebraic geometry and arithmetic statistics. The primary purpose of this paper is to understand a refined distribution of nonblocking curves by asking the following question:

\begin{question}\label{quest:main} Let $\mathcal{L}=\langle F, G\rangle$ be a pencil of plane curves such that $\{F=0\}$ and $\{G=0\}$ have no common $\Fq$-points. Is there a quantitative lower bound for the number of curves in $\mathcal{L}$ defined over $\F_q$ which are nonblocking?
\end{question}

To elaborate further, let $\mathcal{L}=\langle F, G\rangle$ be a pencil that has no $\F_q$-points in its base locus. We refer to Example~\ref{rem:common-Fq-base-point}, which illustrates that the condition on the base locus is natural and necessary. Given such a pencil, consider the $q+1$ curves
$$
C_{[s:t]}=\{sF+tG=0\}
$$
where $[s:t]$ ranges in $\bP^1(\Fq)$. These curves will be called \emph{$\Fq$-members} of $\mathcal{L}=\langle F, G\rangle$. Since $\mathcal{L}$ has no $\Fq$-points in its base locus, the sets of $\Fq$-rational points on these $q+1$ curves are pairwise disjoint and cover all the $\Fq$-points of the plane. Indeed, if $P\in\mathbb{P}^2(\Fq)$ is any $\Fq$-point, it belongs to a unique member $C_{[s:t]}$ with $[s:t]=[-G(P):F(P)]$. In summary, the collection of sets $\{C_{[s:t]}(\Fq)\}_{[s:t]\in\bP^1(\Fq)}$ forms a partition of $\bP^2(\Fq)$ into $q+1$ parts, with the understanding that some of the parts may be empty. We will say that the pencil $\mathcal{L}$ \emph{induces} the partition $\{C_{[s:t]}(\Fq)\}_{[s:t]\in\bP^1(\Fq)}$.

Let us explain a heuristic that suggests a plausible answer to Question~\ref{quest:main}. Since there are $q+1$ distinct $\Fq$-members of $\mathcal{L}$, and together they cover $q^2+q+1$ points, it follows that the average number of $\Fq$-points on a given $\F_q$-member of $\mathcal{L}$ is exactly:
$$
\frac{q^2+q+1}{q+1} < q + 1.
$$
Since there does not exist a blocking set with less than $q+1$ points, it is immediate that $\mathcal{L}$ contains at least one nonblocking curve. Note that the averaging argument does not tell us whether or not the ``median" number of points on a random $\Fq$-member in $\mathcal{L}$ is less than $q+1$. Nevertheless, it is natural to ask whether at least half of the $\Fq$-members in the pencil must be nonblocking. Making this last question slightly weaker, we may instead ask the following.

 \begin{question}\label{quest:positive-prop} Let $\mathcal{L}=\langle F, G\rangle$ be a pencil of plane curves such that $\{F=0\}$ and $\{G=0\}$ have no common $\Fq$-points. Does there exist a universal constant $c_0>0$ such that at least $c_0 q$ distinct $\Fq$-members of $\mathcal{L}$ are nonblocking?
\end{question}

Question~\ref{quest:positive-prop} turns out to have a negative answer in general. 

\begin{thm}\label{thm:positive-proportion-counterex} Let $q=p^n$ be a prime power with $n$ even. There exists a pencil of plane curves over $\F_q$ with no $\F_q$-points in its base locus, which contains only $\sqrt{q}$ many nonblocking curves.
\end{thm}

The key ingredient in the proof of Theorem~\ref{thm:positive-proportion-counterex} is to establish the connection between Question~\ref{quest:positive-prop} and the question of determining the maximum number of disjoint blocking sets in $\bP^2(\F_q)$, first studied by Beutelspacher and Eugeni~\cite{BE86} and independently by Csima and F\"uredi \cite{CF86}. The latter question arises more generally and abstractly in the setting of hypergraph coloring; see \cite{BMPS06} for related discussions. For the discussion of the maximum permissible value of $c_0$ in Question~\ref{quest:positive-prop}, see Remark~\ref{rem:disjoint-blocking-sets}. We will prove in Proposition~\ref{prop:sqrt-lower-bound} that every pencil with no $\Fq$-points in its base locus has at least $\sqrt{q}$ many nonblocking curves, so Theorem~\ref{thm:positive-proportion-counterex} is sharp. 

 Let $p=\operatorname{char}(\mathbb{F}_q)$. We show that Question~\ref{quest:positive-prop} has a positive answer with $c_0=\frac{1}{3}$ when $q=p$ or the degree of the linear system is at most $p$. We thank the anonymous referee for suggesting the proof of this theorem.

\begin{thm}\label{thm:low-degree-pencil}
Given a pencil of plane curves of degree $d$ in $\mathbb{P}^2$ defined over $\mathbb{F}_q$ with no $\Fq$-points in its base locus, assume $q=p$ or $d\leq p$. Then at least $\frac{q+1}{3}$ distinct $\Fq$-members of the pencil are nonblocking curves. 
\end{thm}

In Question~\ref{quest:positive-prop}, the constant $c_0>0$ is required to be universal. If we allow the constant to depend on the degree of the curves, then Question~\ref{quest:positive-prop} has a positive answer. More precisely, we will prove the following effective result.

\begin{thm}\label{thm:fixed-degree-pencil}
Given a pencil of plane curves of degree $d\leq q$ in $\mathbb{P}^2$ defined over $\mathbb{F}_q$ with no $\Fq$-points in its base locus, at least $\frac{q+1}{d+1}$ distinct $\Fq$-members of the pencil are nonblocking curves.
\end{thm}

The inequality $d\leq q$ in the hypothesis is natural. Indeed, if $d>q$, then $0<\frac{q+1}{d+1}<1$, and the conclusion still holds because we have already seen above that the pencil must have at least one nonblocking $\Fq$-member. 

\subsection*{Outline of the paper.} In Section~\ref{sec:combinatorial-reduction}, we show that every partition of the $q^2+q+1$ points of $\bP^2(\F_q)$ into $q+1$ sets can be realized by a pencil of plane curves. This key result Proposition~\ref{prop:combinatorial-reduction} allows us to prove our main Theorem~\ref{thm:positive-proportion-counterex} in Section~\ref{sec:lower-bounds}. We also prove Theorem~\ref{thm:low-degree-pencil} in Section~\ref{sec:lower-bounds}. We discuss the case of fixed-degree pencils and prove Theorem~\ref{thm:fixed-degree-pencil} in Section~\ref{sec:fixed-degree-case}.

\section{Realizing partitions by pencils of curves}\label{sec:combinatorial-reduction}

The goal of this section is to prove the following result, which shows that any partition of $\mathbb{P}^2(\Fq)$ into $q+1$ sets (where some sets could be empty) is induced by a pencil of plane curves.

\begin{prop}\label{prop:combinatorial-reduction}
Suppose $U_1, U_2, \dots, U_{q+1}$ form a partition of $\mathbb{P}^2(\mathbb{F}_q)$. Then there exists a pencil of curves $\mathcal{L}=\langle F, G\rangle$ whose base locus has no $\Fq$-points such that $\mathcal{L}$ induces the partition $\{U_i\}_{i=1}^{q+1}$.
\end{prop}

The proof of Proposition~\ref{prop:combinatorial-reduction} relies on two lemmas. 

\begin{lem}\label{lem:interpolation-1}
Given any $Q\in\bP^2(\Fq)$, there exists a homogeneous polynomial $S_Q\in\Fq[x,y,z]$ of degree $3(q-1)$ such that $S_Q(Q)=1$, while $S_Q(P)=0$ for each point $P\ne Q$ in $\bP^2(\Fq)$. 
\end{lem}

\begin{proof}
Suppose $Q:=[a_0:b_0:c_0]\in \bP^2(\Fq)$. Without loss of generality, we may assume that $c_0\ne 0$. Let $L_1$ and $L_2$ be two (distinct) $\Fq$-lines passing through the point $Q$; we let their equations be $\alpha_1x+\beta_1+\gamma_1z=0$ and $\alpha_2x+\beta_2y+\gamma_2z=0$ for some $\alpha_i,\beta_i,\gamma_i\in\Fq$. Consider the homogeneous polynomial $S_Q$ of degree $3(q-1)$ given by
$$S_Q(x,y,z)=z^{q-1}\cdot \left(z^{q-1}-(\alpha_1x+\beta_1y+\gamma_1z)^{q-1}\right)\cdot \left(z^{q-1}- (\alpha_2x+\beta_2y+\gamma_2z)^{q-1}\right).$$
Since $z$-coordinate of $Q$ is nonzero by assumption, we get $S_Q(Q)=1$. On the other hand, $S_Q(P)=0$ for each point $P\ne Q$ in $\bP^2(\Fq)$. Indeed, given such a point $P=[a_1:b_1:c_1]$, we have two cases: $c_1=0$ and $c_1\neq 0$. If $c_1=0$, then $S_Q(P)=0$ is immediate. If $c_1\neq 0$, then $P$ cannot be on both $L_1$ and $L_2$, and we again obtain $S_Q(P)=0$. 
\end{proof}

\begin{rem}\label{rem:interpolation}
In Lemma~\ref{lem:interpolation-1}, we arranged $S_Q$ to have degree divisible by $q-1$. This allows us to view $S_Q$ as a well-defined function $\mathbb{P}^2(\Fq)\to \Fq$. Indeed, if $[x_0:y_0:z_0]$ and $[x_1:y_1:z_1]$ represent the same point in $\bP^2(\Fq)$, then $(x_1, y_1, z_1) = (\lambda x_0, \lambda y_0, \lambda z_0)$ for some $\lambda\in\Fq^{\ast}$. As a result, $S_Q(x_1,y_1,z_1)=\lambda^{3(q-1)} S_Q(x_0, y_0, z_0) = S_Q(x_0, y_0, z_0)$, confirming that $S_Q$ is a well-defined function on $\bP^2(\Fq)$.
\end{rem}

\begin{lem}\label{lem:interpolation-2}
For any function $f\colon \mathbb{P}^2(\Fq)\to \Fq$ which is not identically zero, there exists a homogeneous polynomial $R_{f}\in \Fq[x,y,z]$ of degree $3(q-1)$ such that $R_f(a,b,c)=f([a:b:c])$ for each $[a:b:c]\in\mathbb{P}^2(\Fq)$.
\end{lem}

\begin{proof}
Borrowing the notation from Lemma~\ref{lem:interpolation-1}, we set
$$
R_f \colonequals \sum_{Q\in \bP^2(\Fq)}f(Q)\cdot S_Q
$$
which satisfies the desired condition.
\end{proof}

We have gathered the tools to prove Proposition~\ref{prop:combinatorial-reduction}. 

\begin{proof}[Proof of Proposition~\ref{prop:combinatorial-reduction}] Given a partition of $\mathbb{P}^2(\Fq)$ into $q+1$ sets $U_1, \ldots, U_{q+1}$, we may assume without loss of generality that $U_1\neq\emptyset$. There exists a function $\varphi\colon \bP^2(\Fq)\to \bP^1(\Fq)$ with the property that the preimages $\varphi^{-1}(Q)$ for the $q+1$ points $Q\in\bP^1(\Fq)$ provide exactly the same partition of $\bP^2(\Fq)$ as $U_1,\ldots U_{q+1}$, where we prescribe $\varphi^{-1}([1:1])=U_1$. We can find two functions $f\colon \bP^2(\Fq)\to \Fq$ and $g\colon \bP^2(\Fq)\to \Fq$ such that for each $P \in \bP^2(\F_q)$, 
\begin{equation}\label{eq:the-map-phi}
\varphi(P) = [f(P):g(P)].
\end{equation}
Note that $f$ and $g$ are not identically zero because $\varphi(P)=[1:1]$ for any $P\in U_1$. By Lemma~\ref{lem:interpolation-2}, we have two homogeneous polynomials $R_f$ and $R_g$ both of degree $3(q-1)$ which induce $f$ and $g$, respectively. 

Let $F=-R_g$ and $G=R_f$. We claim that the pencil $\mathcal{L}=\langle F, G\rangle$ induces the partition $U_1, \ldots U_{q+1}$. First, $\mathcal{L}$ has no $\Fq$-points in its base locus since $f(P)$ and $g(P)$ cannot be zero simultaneously for any $P\in\mathbb{P}^2(\Fq)$ by \eqref{eq:the-map-phi}. Next, for any $[u:v]\in\bP^1$ and a point $P\in\bP^2$, we have
$$
u F(P) + v G(P) = 0 \ \ \Longleftrightarrow \ \ [u:v]=[G(P):-F(P)]=[f(P):g(P)]=\varphi(P).
$$
Thus, a given $\Fq$-point $P$ belongs to the $\Fq$-member of the pencil $\mathcal{L}$ parametrized by $[u:v]\in\bP^1(\Fq)$ if and only if $P\in \varphi^{-1}([u:v])$. Consequently, the sets of $\Fq$-points contained in the $q+1$ distinct $\Fq$-members of the pencil $\mathcal{L}$ precisely correspond to the partition $U_1, \ldots, U_{q+1}$. \end{proof}

\begin{rem}\label{rmk:prop-general} It is possible to generalize Proposition~\ref{prop:combinatorial-reduction} where $U_1, U_2, ..., U_{q+1}$ together cover $\mathbb{P}^2(\Fq)$, although they are no longer assumed to be pairwise disjoint. To obtain a pencil of curves that induces the sets $U_1, U_2, ..., U_{q+1}$, it is necessary that the equality
$$U_i\cap U_j = B \colonequals \bigcap_{k=1}^{q+1} U_k$$
is satisfied for each $i\ne j$. Under this assumption, one can prove that there exists a pencil $\mathcal{L}$ of plane curves with $\Fq$-members $C_1, C_2, ..., C_{q+1}$ such that $C_i(\Fq) = U_i$ for each $1\leq i\leq q+1$. The proof is almost identical to the current proof of Proposition~\ref{prop:combinatorial-reduction}. We start by picking
$$
\varphi\colon \bP^2(\Fq)\setminus B \longrightarrow \bP^1(\Fq)
$$
so that the preimages $\varphi^{-1}(\alpha_i)$ of the points $\alpha_1, ..., \alpha_{q+1}$ in $\bP^1(\Fq)$ coincide with the pairwise disjoint sets $U_i\setminus B$ for $i=1,..,q+1$. By using Lemma~\ref{lem:interpolation-2}, we can find homogeneous polynomials $F$ and $G$ which both vanish on $B$ such that the pencil $\mathcal{L}=\langle F, G\rangle$ satisfies the desired properties.
\end{rem}

\section{Lower bounds on the number of nonblocking curves}\label{sec:lower-bounds}

The averaging argument presented in the introduction showed that if a pencil of plane curves has no $\Fq$-points in its base locus, there must be at least one nonblocking curve in the pencil. We explain what happens when the hypothesis on the base locus is dropped. If the set of $\Fq$-points of the base locus is itself a blocking set, then clearly, all the curves in the pencil are blocking. Additional examples of pencils whose $\Fq$-members are all blocking can be constructed using the explicit families of blocking curves in \cite{AGY23}*{Proposition 6.1}; in such pencils, the base locus may have a large number of $\Fq$-points. It may be tempting to believe there should be at least some nonblocking curves if the base locus is small. However, we will now exhibit a pencil $\mathcal{L}$ of plane curves of degree $d\geq 2$ over $\mathbb{F}_q$ containing \emph{exactly one} $\Fq$-point in its base locus such that every $\Fq$-member of $\mathcal{L}$ is a blocking curve. This example shows that the hypothesis that the base locus has no $\Fq$-points cannot be relaxed.

\begin{ex}\label{rem:common-Fq-base-point} 
Let $m\in\mathbb{N}$ with $\gcd(m, q-1)=1$. Consider the pencil $\mathcal{L}=\langle F, G\rangle$, where
\begin{align*}
F(x,y,z) = y^m, \quad\quad G(x,y,z) = z^m.
\end{align*}
The base locus of this pencil, defined by the intersection of $\{F=0\}$ and $\{G=0\}$, consists of a single $\Fq$-point $[1:0:0]$. Now, consider an arbitrary curve $C_{[a:b]} =\{aF+bG=0\}$ in the pencil, where $[a:b]\in\mathbb{P}^1(\mathbb{F}_q)$. We show that the curve $C_{[a:b]}=\{aF + bG = 0\}$ intersects \emph{every} $\Fq$-line $L$ at some $\Fq$-point. 

If $L$ is defined by $\lambda y + \mu z=0$ for some $\lambda,\mu\in\mathbb{F}_q$, then $[1:0:0]$ is a desired intersection point. Otherwise, suppose $L$ is given by $x=\lambda y + \mu z$ for some $\lambda,\mu\in\mathbb{F}_q$. Take any solution $[y_0\colon z_0]\in\mathbb{P}^1(\mathbb{F}_q)$ to $ay^m+bz^m=0$. Such a solution exists, since $m$ was chosen so that $\gcd(m,q-1) = 1$; the map
$t\mapsto t^{m}$ is an automorphism of the cyclic group $(\Fq^*,\cdot)$. The point $[\lambda y_0+\mu z_0\colon y_0\colon z_0]$ lies in the intersection of the line $L$ and the curve $ay^m+bz^m=0$. Consequently, every $\Fq$-member of $\mathcal{L}$ is a blocking curve. Alternatively, one can check that each $\Fq$-member of $\mathcal{L}$ has an $\F_q$-linear component and is thus blocking. 

An alternative construction can be carried out by using the generalized version of Proposition~\ref{prop:combinatorial-reduction} (see Remark~\ref{rmk:prop-general}) where $U_1, U_2, ..., U_{q+1}$ consist of $\Fq$-points of $q+1$ distinct lines passing through a common point. The pencil produced using that method will have a degree $3(q-1)$.
\end{ex}

We have seen that having even a single $\Fq$-point in the base locus may cause all of the $\Fq$-members to be blocking curves. Thus, for the remainder of the paper, we will work under the hypothesis that there are no $\Fq$-points in the base locus of a given pencil. Our first quantitative lower bound on the number of nonblocking curves is implicit in the finite geometry literature \cite{BE86}*{Theorem 2.2} and \cite{CF86}*{Theorem 4.1}. We include a simple proof for completeness; the ideas and notation used in this proof will also appear in the remarks after the proof. 

\begin{prop}\label{prop:sqrt-lower-bound}
Every pencil of plane curves over $\Fq$ with no $\Fq$-points in its base locus has at least $\sqrt{q}$ many $\Fq$-members that are nonblocking.
\end{prop}

\begin{proof}
Let $\mathcal{L}=\langle F, G\rangle$ be a pencil of plane curves with no $\Fq$-points in its base locus. If $\mathcal{L}$ has an $\Fq$-member $C$ whose $\Fq$-rational points form a trivial blocking set, then there exists an $\Fq$-line $L$ such that $L(\Fq)\subseteq C(\Fq)$. In this case, the other $q$ members of $\mathcal{L}$ will be nonblocking by being disjoint from $L(\Fq)$. Thus, we may assume that each $\Fq$-member of the pencil is not trivially blocking. 

Suppose we have $m$ blocking and $q+1-m$ nonblocking $\Fq$-members of $\mathcal{L}$. Since a nontrivial blocking set has at least $q+\sqrt{q}+1$ points~\cite{B71}, the following inequality holds:
$$
m(q+\sqrt{q}+1) \leq  \sum_{\substack{C \text{ is } \Fq\text{-member of } \mathcal{L} \\ C\text{ is blocking}}} \# C(\Fq)\leq \sum_{C \text{ is } \Fq\text{-member of } \mathcal{L}} \# C(\Fq) = q^2+q+1.
$$
It follows that,
$$
m \leq \frac{q^2+q+1}{q+\sqrt{q}+1} = q-\sqrt{q}+1. 
$$
Thus, the number of nonblocking $\Fq$-members is $q+1-m\geq\sqrt{q}$, as desired. 
\end{proof}

As a complement to the previous proposition, we now present a quick proof of Theorem~\ref{thm:positive-proportion-counterex}, which guarantees the existence of a pencil with exactly $\sqrt{q}$ many nonblocking members. Recall that if $q$ is a square, a {\em Baer subplane} in $\bP^2(\F_q)$ is a subplane with size $q+\sqrt{q}+1$. It is well-known that Baer subplanes are blocking sets.

\begin{proof}[Proof of Theorem~\ref{thm:positive-proportion-counterex}] Since $q$ is a square, it is known that that $\bP^2(\F_q)$ can be partitioned into 
$$
\frac{q^2+q+1}{q+\sqrt{q}+1}=q-\sqrt{q}+1
$$
Baer subplanes \cite{H79}*{Theorem 4.3.6}. Now, Proposition~\ref{prop:combinatorial-reduction} implies that we can find a pencil of curves with exactly $\sqrt{q}$ many $\Fq$-members that are nonblocking.
\end{proof}

\begin{rem}\label{rem:disjoint-blocking-sets}
For each prime power $q$, one can construct $(1/3-o(1))q$ disjoint blocking sets in $\bP^2(\F_q)$ \cite{BMPS06}. By Proposition~\ref{prop:combinatorial-reduction}, there exists a pencil in which at least $1/3$ of its members are blocking, or equivalently, at most $2/3$ of its members are nonblocking. It is conjectured by Kriesell that $\bP^2(\F_q)$ can be partitioned into $\lfloor q/2 \rfloor$ blocking sets (he verified the conjecture for $q \leq 8$) \cite{BMPS06}*{Page 150}, so the best possible constant in Question~\ref{quest:positive-prop} is conjecturally $1/2$.
\end{rem}

Next, we prove Theorem~\ref{thm:low-degree-pencil}, which gives a positive answer to Question~\ref{quest:positive-prop} under certain conditions related to the characteristic of the finite field.

\begin{proof}[Proof of Theorem~\ref{thm:low-degree-pencil}]

If $C$ is an $\mathbb{F}_q$-member of the pencil such that $C(\mathbb{F}_q)$ forms a trivial blocking set, then the other $q$ members are nonblocking by the argument in the first paragraph of the proof of Proposition~\ref{prop:sqrt-lower-bound}. We may therefore assume that each $\mathbb{F}_q$-member is either nonblocking or nontrivially blocking.

We claim that each blocking curve in the pencil has at least $\frac{3}{2}(q+1)$ distinct $\mathbb{F}_q$-points. This proves the theorem, because the number of blocking curves in the pencil is at most $(q^2+q+1)/\frac{3}{2}(q+1)<\frac{2}{3}(q+1)$ by the same counting argument as in the proof of Proposition~\ref{prop:sqrt-lower-bound}.

When $q=p$, the claim follows from Blokhuis' theorem, which states that a nontrivial blocking set in $\bP^2(\F_p)$ has at least $\frac{3}{2}(p+1)$ points \cite{B94}. Next, assume that $d\leq p$. Suppose, to the contrary, that we have a blocking curve $H$ in the pencil with $|H(\F_q)|<\frac{3}{2}(q+1)$. Then $H(\mathbb{F}_q)$ contains a minimal blocking set $B$. By a result of Sz\H{o}nyi \cite{S97}*{Corollary 4.8}, each line intersects $B$ in $1$ modulo $p$ points. If we take a line $L$ that passes through two points of $B$, it must intersect $B$ in at least $p+1$ points. In particular, $|(L\cap H)(\mathbb{F}_q)|\geq |L(\F_q) \cap B)|\geq p+1$. This contradicts B\'ezout's theorem as the degree of the curve $H$ is $d\leq p$.    
\end{proof}

\begin{rem}
When $q$ is a square, we have seen that the number of nonblocking curves in a pencil must be at least $\sqrt{q}$, and this lower bound is sharp in Theorem~\ref{thm:positive-proportion-counterex}. When $q$ is not a square, we know each nontrivial blocking set has a size at least $q+1+cq^{2/3}$ \cite{BSS99}, so the same argument as in Proposition~\ref{prop:sqrt-lower-bound} can be adapted to show that there are at least $c'q^{2/3}$ many nonblocking $\Fq$-members in any pencil with no $\Fq$-points in its base locus. This can be seen by mimicking the same proof as in Proposition~\ref{prop:sqrt-lower-bound} to get the following upper bound on the number of blocking curves in the given pencil:
$$
\frac{q^2+q+1}{q+1+q^{2/3}} = \frac{q^2+q-q^{4/3}}{q+1+q^{2/3}} + \frac{q^{4/3}}{q+1+q^{2/3}} \approx q+1-q^{2/3}+\frac{q^{4/3}}{q+1+q^{2/3}} \approx q+1-q^{2/3}+q^{1/3}.
$$

More generally, if the size of the smallest nontrivial blocking set in $\bP^2(\F_q)$ is $q+1+c q^{\epsilon}$, then a similar argument shows that there are at least $c' q^{\epsilon}$ nonblocking $\mathbb{F}_q$-members in the pencil. The minimum size of a nontrivial blocking set in $\bP^2(\F_q)$ is believed to depend on $n$, where $q=p^n$ with $n\geq 2$. It is a folklore conjecture that the lower bound for the size of a nontrivial blocking set is $q+1+q/p^e$, where $e$ is the largest proper divisor of $n$. We refer to \cites{B96, LP00} for further discussion.
\end{rem} 

\begin{rem}\label{rem:generic}
Sz\H{o}nyi \cite{S97}*{Theorem 5.7} showed if $q=p^2$ with $p$ prime, a nontrivial blocking set in $\bP^2(\F_q)$ which does not contain Baer subplanes has size at least $3(q+1)/2$. Note that this result is analogous to Blokhuis' result over $\F_p$ with $p$ prime. Let us explain how this observation helps us answer Question~\ref{quest:positive-prop} affirmatively in the case $q=p^2$ for certain pencils.

We say that a pencil of curves over $\Fq$ is \emph{generic} if the number of singular members (defined over the algebraic closure $\overline{\mathbb{F}_q}$) is finite. The terminology arises from the fact that a generic line (over $\overline{\Fq}$) in the parameter space of plane curves of degree $d$ meets the discriminant hypersurface in finitely many points. It follows that a given pencil $\mathcal{L}$ is either generic or every $\overline{\Fq}$-member of $\mathcal{L}$ is a singular curve. Thus, non-generic pencils are extremely special. When $d=o(\sqrt{q})$, we can show that for any generic pencil, at least $(1/3-o(1))q$ members are nonblocking. We may assume that no curve in the pencil is trivially blocking, for otherwise, $q$ of the curves are nonblocking. For each generic pencil there are at most $3(d-1)^2=o(q)$ many singular curves by \cite{EH16}*{Proposition 7.4}. Note for each smooth curve in the pencil, by B\'ezout's theorem, it does not contain a Baer subplane (otherwise, there is a line that intersects at the curve with at least $\sqrt{q}+1>d$ points). Thus, if a smooth member in the pencil is blocking, it has a size of at least $3(q+1)/2$; therefore, our pencils have at most $2q/3$ smooth blocking curves. Hence, we have more than $q/3-3(d-1)^2=(1/3-o(1))q$ nonblocking curves in our pencil.
\end{rem}

\section{Proof of Theorem~\ref{thm:fixed-degree-pencil}} \label{sec:fixed-degree-case}

In this section, we will show that Question~\ref{quest:positive-prop} has a positive answer if the constant $c_0$ is allowed to depend on $d$. More precisely, we will prove Theorem~\ref{thm:fixed-degree-pencil}, which guarantees that at least $\frac{1}{d+1}$ fraction of the $\Fq$-members of a given pencil are nonblocking. 

As a preparation, we prove a lower bound on the number of $\Fq$-rational points on a blocking plane curve. The following lemma is implicitly contained in \cite{AGY23}*{Section 2.5}. For the sake of completeness, we present a self-contained proof.

\begin{lem}\label{lem:lower-bound-blocking-curve} Suppose $C$ is a plane curve of degree $d$ defined over $\Fq$. If $C(\Fq)$ is a nontrivial blocking set, then $\# C(\Fq) > q + \frac{q+\sqrt{q}}{d}$.
\end{lem}

\begin{proof}
Let $t_i$ denote the number of $\Fq$-lines $L$ such that $C\cap L$ has exactly $i$ distinct $\Fq$-points. Let $N=\#C(\Fq)$ denote the number of $\Fq$-points of $C$. Since $C$ is nontrivially blocking, it cannot contain any $\Fq$-line as a component. Thus, $t_0=0$, and $t_i=0$ for $i>d$ by B\'ezout's theorem. Moreover, a standard double counting argument on point-line incidences (see, for example \cite{AGY23}*{Lemma 2.9}) leads to the following three identities:
\begin{align*}
\sum_{i=1}^{d} t_i = q^2+q+1, \quad\quad \sum_{i=1}^{d} i\cdot t_i = (q+1) N, \quad\quad \sum_{i=2}^{d} \binom{i}{2}\cdot t_i =\binom{N}{2}.
\end{align*}
By combining the first two identities, we obtain
\begin{equation}\label{ineq:incidence-1}
q^2+q+1 = (q+1)N - \sum_{i=2}^{d} (i-1)t_i,
\end{equation}
while the third identity implies
\begin{equation}\label{ineq:incidence-2}
\sum_{i=2}^d (i-1)t_i\ge \frac{N(N-1)}{d}.
\end{equation}
The inequalities \eqref{ineq:incidence-1} and \eqref{ineq:incidence-2} together yield,
\begin{equation}\label{ineq:incidence-3}
q^2+q+1 \leq (q+1)N - \frac{N(N-1)}{d}.
\end{equation}
Since $C(\Fq)$ is a \emph{nontrivial} blocking set, we have $N\geq q+\sqrt{q}+1$ by \cite{B71}. Now, the inequality \eqref{ineq:incidence-3} implies,
\begin{align*}
(q+1)N & \geq q^2+q+1 + \frac{(q+\sqrt{q}+1)(q+\sqrt{q})}{d} > (q+1)q + \frac{(q+1)(q+\sqrt{q})}{d}. 
\end{align*}
It follows that $N>q+\frac{q+\sqrt{q}}{d}$, as desired. \end{proof}

\begin{rem}
One can obtain a slightly stronger conclusion $N > \left(\frac{d}{d-1}\right)q\cdot\left(1+o(1)\right)$ by analyzing the inequality \eqref{ineq:incidence-3} more carefully. 
\end{rem}

We now proceed with the proof of Theorem~\ref{thm:fixed-degree-pencil} on the number of nonblocking curves in a pencil of fixed degree.

\begin{proof}[Proof of Theorem~\ref{thm:fixed-degree-pencil}] Given a pencil $\mathcal{L}$ with no $\Fq$-points in its base locus, suppose $m$ of the $\Fq$-members are blocking and $q+1-m$ of the $\Fq$-members are nonblocking. If any $\Fq$-member of $\mathcal{L}$ is trivially blocking, then applying the same argument as at the beginning of Proposition~\ref{prop:sqrt-lower-bound}, we see that the other $q$ members of the pencil are nonblocking. So, we may assume that all the blocking $\Fq$-members are nontrivially blocking.

By applying Lemma~\ref{lem:lower-bound-blocking-curve}, we obtain the following inequality:
$$
m \left(\frac{qd+q+\sqrt{q}}{d}\right) < \sum_{\substack{C \text{ is } \Fq\text{-member of } \mathcal{L} \\ C\text{ is blocking}}} \# C(\Fq)\leq \sum_{C \text{ is } \Fq\text{-member of } \mathcal{L}} \# C(\Fq) = q^2+q+1.
$$
Therefore, using the hypothesis $d \leq q$, we obtain
$$
m \leq \frac{(q^2+q+1)d}{qd+q+\sqrt{q}}=\frac{(q+1)(q+\frac{1}{q+1})d}{(d+1)(q+\frac{\sqrt{q}}{d+1})}\leq \frac{d}{d+1} (q+1).
$$
We conclude that the number of nonblocking $\Fq$-members of the pencil $\mathcal{L}$ is $q+1-m\geq \frac{q+1}{d+1}$. \end{proof}

\begin{rem} The hypothesis $d\leq q$ in Theorem~\ref{thm:fixed-degree-pencil} is natural. However, when $\sqrt{q}\leq d\leq q$, Theorem~\ref{thm:fixed-degree-pencil} only guarantees $\frac{q+1}{\sqrt{q}+1}=\sqrt{q}-1$ many nonblocking curves, which was already proved in Proposition~\ref{prop:sqrt-lower-bound}. For Theorem~\ref{thm:fixed-degree-pencil} to yield more refined results, we need $d<\sqrt{q}$. 

On the other hand, if $d$ is too small relative to $q$, namely, $d<q^{1/6}$, then we can refine Theorem~\ref{thm:fixed-degree-pencil} for generic pencils (as defined in Remark~\ref{rem:generic}). More precisely, when $q>d^6$, the lower bound in Theorem~\ref{thm:fixed-degree-pencil} can be improved significantly so that the answer to Question~\ref{quest:positive-prop} is positive with $c_0=1-o(1)$. This follows from the observation that the number of singular members in a generic pencil is at most $3(d-1)^2=o(q)$ as stated in Remark~\ref{rem:generic}, and from our earlier result that a smooth plane curve is not blocking whenever $q>d^6$ \cite{AGY23}*{Theorem 1.2}. There are also other refined sufficient conditions on nonblocking smooth curves in \cite{AGY23}.
\end{rem}

\section*{Acknowledgments}
The second author is supported by an NSERC Discovery grant. The authors would like to thank the anonymous referees for their valuable comments and suggestions, and in particular for pointing out the proof of Theorem~\ref{thm:low-degree-pencil}.

\bibliographystyle{abbrv}
\bibliography{biblio}

\end{document}